\documentclass[5p]{elsarticle}
\usepackage[utf8]{inputenc}
\usepackage[english]{babel}
\usepackage{amsfonts,amsmath,amssymb,amsthm}
\usepackage[T1]{fontenc}
\usepackage{url}
\usepackage{hyperref}
\usepackage{color}
\usepackage{tabularx}
\usepackage{faktor}
\usepackage{array}

\definecolor{szin}{rgb}{0,0.44,0.4}
\definecolor{szin2}{rgb}{0.902,0.2705,0}
\definecolor{szin3}{rgb}{0,0.5,0}

\hypersetup{pdfpagemode=UseNone, pdftoolbar=true, breaklinks=true,
colorlinks=true,
linkcolor=szin2,
linktoc=page,
citecolor=szin3,bookmarks=false}
\makeindex

\usepackage{graphicx} 
\graphicspath{{./plots/}}

\newtheorem{theorem}{Theorem}

\newtheorem{proposition}{Proposition}

\newcommand{\pitau}{{(\pi,\tau)}}
\newcommand{\pioverlinetauoverline}{{(\overline{\pi},\overline{\tau})}}
\newcommand{\pitauoverline}{{(\pi,\overline{\tau})}}
\newcommand{\taupi}{{(\tau,\pi)}}

\newcommand{\GP}[1]{\mathcal{GP}_{#1}}

\newcolumntype{L}[1]{>{\raggedright\let\newline\\\arraybackslash\hspace{0pt}}m{#1}}
\newcolumntype{C}[1]{>{\centering\let\newline\\\arraybackslash\hspace{0pt}}m{#1}}
\newcolumntype{R}[1]{>{\raggedleft\let\newline\\\arraybackslash\hspace{0pt}}m{#1}}

\begin{document}

\begin{frontmatter}

\title{Small cycles, generalized prisms and Hamiltonian cycles in the Bubble-sort graph} 

\author[evk]{Elena V. Konstantinova}
\ead{e\_konsta@math.nsc.ru}
\address[evk]{Sobolev Institute of Mathematics, Novosibisk State University, Novosibirsk, Russia}

\author[anm]{Alexey N. Medvedev\corref{cor1}}
\ead{an\_medvedev@yahoo.com}
\address[anm]{ICTEAM, Universit\'{e} catholique de Louvain, Louvain-la-Neuve, Belgium}
\cortext[cor1]{Corresponding author}

\begin{abstract} 

The Bubble-sort graph $BS_n,\,n\geqslant 2$, is a Cayley graph over the symmetric group $Sym_n$ generated by transpositions from the set $\{(1 2), (2 3),\ldots, (n-1\, n)\}$. It is a bipartite graph containing all even cycles of length $\ell$, where $4\leqslant \ell\leqslant n!$. We give an explicit combinatorial characterization of all its $4$- and $6$-cycles. Based on this characterization, we define generalized prisms in $BS_n,\,n\geqslant 5$, and present a new approach to construct a Hamiltonian cycle based on these generalized prisms.

\end{abstract}


\end{frontmatter}

\section{Introduction}\label{Fact}

We start by introducing a few definitions. The Bubble-sort graph $BS_n=Cay(Sym_n, \mathcal{B}),\,n\geqslant 2$, is a Cayley graph over the symmetric group $Sym_n$ of permutations $\pi=[\pi_1 \pi_2 \ldots \pi_n]$, where $\pi_i=\pi(i)$, $ 1\leqslant i \leqslant n$, with the generating set $\mathcal{B}=\{b_i \in Sym_n: 1 \leqslant i \leqslant n-1\}$ of all bubble-sort transpositions $b_i$ swapping the $i$-th and $(i+1)$-st elements of a permutation $\pi$ when multiplied on the right, i.e. $[\pi_1 \pi_2 \ldots \pi_{i-1} \pi_i \pi_{i+1} \ldots \pi_n] b_i=[\pi_1 \pi_2 \ldots \pi_{i-1}  \pi_{i+1} \pi_i  \ldots \pi_n]$. It is a connected bipartite $(n-1)$--regular graph of order $n!\,$ and diameter $diam(BS_n)={n \choose 2}$. Since this graph is bipartite it does not contain odd cycles but it contains all even $l$--cycles where $l=4, \ldots, n!$~\cite{KA06}. The hamiltonicity of this graph also follows from results in~\cite{J63,KL75,T62}. 

The graph $BS_n,\,n\geqslant 3$, is constructed from $n$ copies $BS_{n-1}(i),\,1\leqslant i \leqslant n$, with the vertex set $\{[\pi_1\pi_2\dots\pi_{n-1}i]\}$, and vertices between copies are connected with external edges from the set $\{\{[\pi_1\pi_2\dots\pi_{n-1} i],[\pi_1\pi_2\dots i \pi_{n-1}]\},\,1\leqslant i\leqslant n\}$, where $\pi_k\in \{1,\dots,n\}\backslash\{i\}$, $1\leqslant k \leqslant n-1$. For example, $BS_3$ is isomorphic to a $6$--cycle, and $BS_4$ is constructed from $4$ copies of $BS_3\cong C_6$ (see Figure~\ref{fig.bs3_bs4}).

The generating set $\mathcal{B}$ plays an important role in computer science for generating all permutations~\cite{J63,K,T62}, in the theory of polytopes where $BS_n$ is considered as $1$-skeleton of the Permutahedron~\cite{T06}, in knot theory~\cite{BB05} and the theory of reflection groups and Coxeter groups~\cite{H90,LS01,O91}. It is known that all transpositions $b_i=(i\ i+1)$,  $1\leqslant i \leqslant n-1$, from the set $\mathcal{B}$ can be considered as $(n-1)$ braids on $n$ strands forming the braid group $B_n$. There is a surjective group homomorphism $B_n \rightarrow Sym_n$ from the braid group into the symmetric group so that the Coxeter presentation of the symmetric group is given by the following $\binom{n}{2}$ braid relations~\cite{KT08}:

\begin{equation}\label{ebraid6}
b_ib_{i+1}b_i=b_{i+1}b_ib_{i+1},  1\leqslant i \leqslant n-2;
\end{equation}
\begin{equation}\label{ebraid4}
b_jb_i=b_ib_j, 1\leqslant i < j-1 \leqslant n-2.
\end{equation}

It is known (see, for example \S1.9 in~\cite{H90} or~\cite{LS01}) that if any permutation $\pi$ is presented by two minimal length expressions $u$ and $v$ in terms of generating elements from the set $\mathcal{B}$, then $u$ and $v$ can be transformed to each other using only the braid relations~(\ref{ebraid6}) and~(\ref{ebraid4}). However, these relations are not the only set and not the smallest set for generating $B_n$.

\begin{figure*}[t]
\centering
\includegraphics[width = 0.9\textwidth]{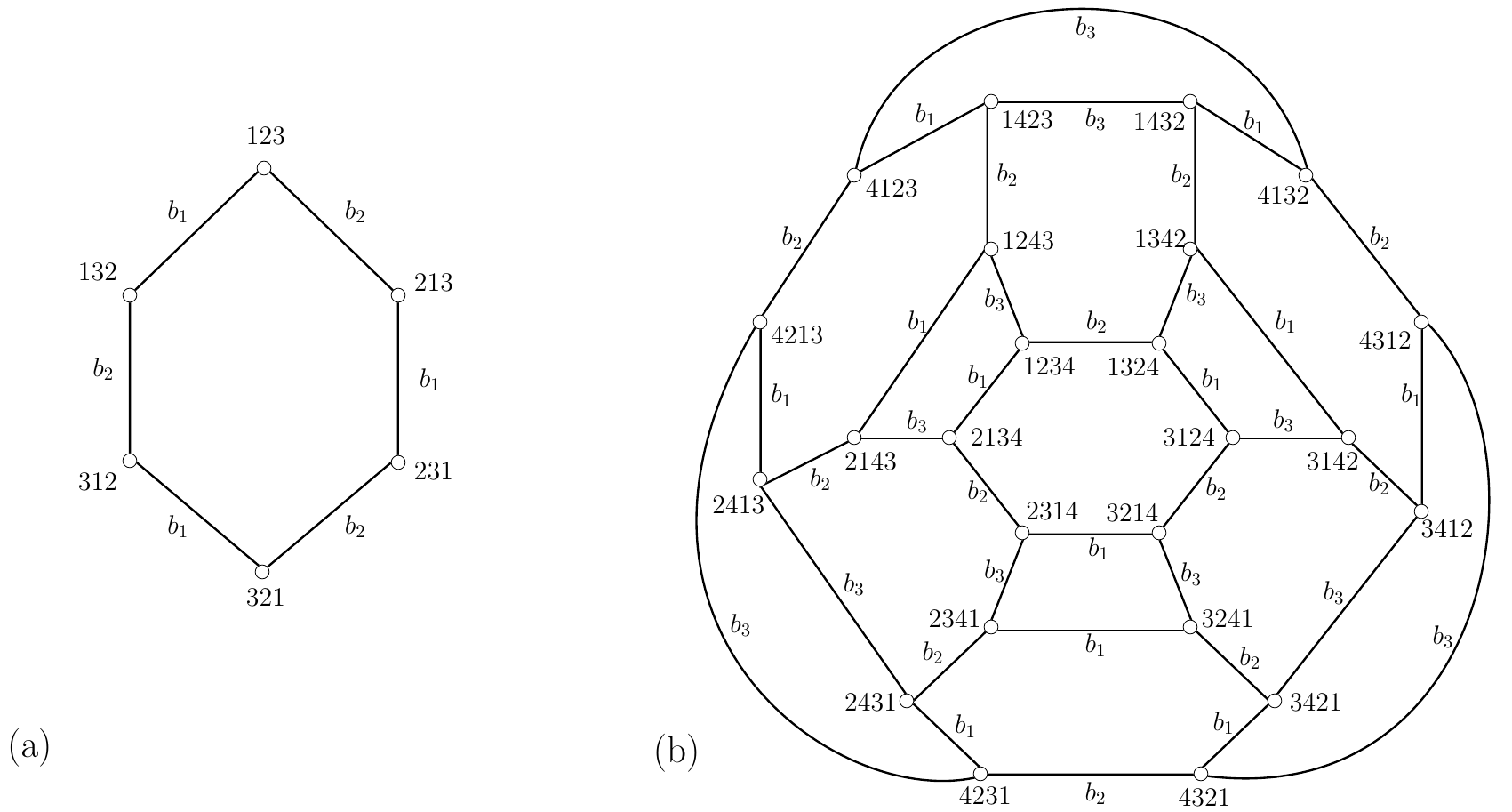}
\caption{Examples of the Bubble-Sort graphs (a) $BS_3$ and (b) $BS_4$}
\label{fig.bs3_bs4}
\end{figure*}

Let us note that the relations~(\ref{ebraid6}) and~(\ref{ebraid4}) form a $6$-cycle and a $4$-cycle in the Bubble-sort graph $BS_n$ for any $n\geqslant 3$ and for any $n\geqslant 4$, correspondingly. As it follows below, we use a cycle description presented in~\cite{KM10} to characterize small cycles in the Pancake graph, and then applied in the Star graph \cite{KM14}.

Two simple paths in a graph are called \textit{non-intersecting}, if they have no common vertices. A sequence of transpositions $C_{\ell}= b_{i_0}\,\ldots\, b_{i_{\ell-1}}$, where $1\leqslant i_j \leqslant n-1$, and $i_j \neq i_{j+1}$ for any $0\leqslant j \leqslant \ell-1$, such that $\pi\,b_{i_0}\ldots b_{i_{\ell-1}}=\pi$, where $\pi \in Sym_n$, is said to be {\it a form of a cycle $C_\ell$ of length $\ell$} in the Bubble-sort graph. A cycle $C_{\ell}$ of length ${\ell}$ is called an ${\ell}$--cycle.  It is evident that any ${\ell}$--cycle can be presented by $2\,{\ell}$ forms (not necessarily different) with respect to a vertex and a direction. {\it The canonical form $C_{\ell}$ of an ${\ell}$--cycle} is called a form with a lexicographically maximal sequence of indices $i_0\ldots i_{\ell-1}$. For cycles of a form $C_{\ell}=b_a b_b\cdots b_a b_b$, where $\ell=2\,k$, and $b_a b_b$ appears $k$ times, we write $C_{\ell}=(b_a b_b)^k$. In particular, $BS_3\cong C_6$ with the following canonical form:

\begin{equation}\label{e1}
C_6=b_2b_1b_2b_1b_2b_1=(b_2\,b_1)^3.
\end{equation}

This representation admits the braid relation~(\ref{ebraid6}) for $n=3$. In the case when $n=4$, for $6$-cycles there are two canonical forms  $C_6=b_{i+1}b_ib_{i+1}b_ib_{i+1}b_i=(b_{i+1}\,b_i)^3$, where $i=1,2$, meeting the same braid relations~(\ref{ebraid6}). Moreover, for $4$-cycles there is the canonical form $C_4=b_3b_1b_3b_1 \\ =(b_3\,b_1)^2$ which admits the braid relation~(\ref{ebraid4}). It is evident from the Figure~\ref{fig.bs3_bs4}, (b), that the $4$- and $6$-cycles have no other canonical representation that is different from the braid relations in $BS_4$. The natural questions arising here is about existence of canonical forms of small cycles not admitting the braid relations in the Bubble-sort graph $BS_n$ for any $n\geqslant 4$.

In this paper we give complete combinatorial characterization of $4$- and $6$-cycles, which is presented in the following two theorems.

\begin{theorem} \label{th1}
Each of vertices of $BS_n, n\geqslant 4$, belongs to $(n-2)(n-3)/2$ distinct $4$--cycles of the following canonical form:
\begin{equation}\label{eq.c4}
C_4=(b_j\,b_i)^2, \quad 1\leqslant i \leqslant n-3,\, i+2\leqslant j \leqslant n-1.
\end{equation}
In total, there are $\frac{(n-2)(n-3)n!}{8}$ distinct $4$--cycles in $BS_n$.
\end{theorem}

Instead of explicitly writing two index conditions as in Equation~\eqref{eq.c4}, for the sake of simplicity we use the compound version, given as $1\leqslant i \leqslant j-2 \leqslant n-3$.

\begin{theorem} \label{th2}
Each of vertices of $BS_n$, $n\geqslant 3$, belongs to $(n-2)$ distinct $6$--cycles of the canonical form:
\begin{equation}\label{eq.c6_1}
C^1_6=(b_i\,b_{i+1})^3, \quad 1\leqslant i \leqslant n-2,\,\,n\geqslant 3;
\end{equation}
when $n\geqslant 5$, each vertex belongs to $3(n-2)(n-3)$ distinct $6$--cycles of the canonical form:
\begin{multline}\label{eq.c6_2}
C^2_6=b_j\,b_{i+1}\,b_i\,b_j\,b_i\,b_{i+1}, \quad 1\leqslant i \leqslant j-3 \leqslant n-4, \text{or } \\ 1\leqslant j \leqslant i-2 \leqslant n-4 ;
\end{multline}
when $n\geqslant 6$, each vertex belongs to $\frac{(n-3)(n-4)(n-5)}{2}$ distinct $6$--cycles of the canonical form:
\begin{equation}\label{eq.c6_3}
C^3_6=(b_k\,b_j\,b_i)^2, \quad 1\leqslant i \leqslant j-2 \leqslant k-4\leqslant n-5;
\end{equation}
\begin{equation}\label{eq.c6_4}
C^4_6=b_k\,b_j\,b_i\,b_k\,b_i\,b_j, \quad 1\leqslant i \leqslant j-2 \leqslant k-4\leqslant n-5;
\end{equation}
\begin{equation}\label{eq.c6_5}
C^5_6=b_k\,b_j\,b_k\,b_i\,b_j\,b_i, \quad 1\leqslant i \leqslant j-2 \leqslant k-4\leqslant n-5;
\end{equation}
\begin{equation}\label{eq.c6_6}
C^6_6=b_k\,b_i\,b_k\,b_j\,b_i\,b_j, \quad 1\leqslant i \leqslant j-2 \leqslant k-4\leqslant n-5.
\end{equation}
In total, there are $(2n^3-21n^2+80n-104)\frac{n!}{6}$ distinct $6$--cycles in $BS_n$.
\end{theorem}

Hamiltonian cycles are one of the main objects of study in graph theory, and without doubt they also have numerous applications in various fields of science. The question of existence is as well important as the proposal of different constructions of these cycles, which is related to finding new Gray codes in combinatorial graphs \cite{KM16} or various embeddings \cite{SNL2011}. In this paper, we propose a new algorithm for constructing the Hamiltonian cycle, which is based on a maximal cover of the graph with generalized prisms. For any graph $H$, the generalized prism, denoted as 2-$H$, is defined as the Cartesian product $H\times K_2$ \cite{PR91}. The \textit{maximal cover} of a graph $G=(V, E)$ by subgraphs $H\subset G$ is a vertex-disjoint embedding of a collection of $H_1,H_2,\dots, H_k$, where $H_i\cong H$ for each $i$, such that each vertex $v\in V$ belongs to some $H_i$. The Hamiltonian cycle based on a maximal cover $H_1,H_2,\dots, H_k$ is the cycle that is formed by finding appropriate Hamiltonian paths in each copy of $H_i$ and fastening them together by the edges between these subgraphs to form the closed Hamiltonian cycle in the original graph. Similar idea was successfully used to construct Hamiltonian cycles in vertex-transitive graphs (see e.g. \cite{KM16}, \cite{Alspach89}, \cite{Kutnar09}). Our result is presented in the following theorem.

\begin{theorem} \label{th3}
In graph $BS_n,\,n\geqslant 5$, there exists a Hamiltonian cycle based on the maximal cover by generalized prisms 2-$BS_{n-2}$.
\end{theorem}

Although we prove the existence of the cycle, the proof is constructive and the construction suggests there are multiple Hamiltonian cycles of the presented type.  

In the literature we could find three proposed Hamiltonian cycle constructions in the Bubble-sort graph. The famous Steinhaus-Johnson-Trotter algorithm \cite{J63,T62} produces a Hamiltonian cycle, which is merely based on the permutation structure, rather than the structure of the graph, and it is different from our algorithm. T. Manneville and V. Pilaud in~\cite{Manneville14} proposed a similar method of constructing a Hamiltonian cycle in a general family of $1$-skeletons of graph associahedra. However their method, in simple words, was based on the hierarchical structure, thus using $b_{n-1}$-edges to connect the chunks of a Hamiltonian cycle in the copies of $BS_{n-1}$. Another construction that also utilizes the fastening method was proposed in~\cite{El-Hashash09}. The authors used a lexicographical ordering of permutations in $BS_{n}$ to obtain subgraphs isomorphic to $BS_{n-1}$ and by connecting them through $4$-cycles also get a Hamiltonian cycle. All the proposed methods are different from our algorithm and we provide the detailed explanation using graph-theoretic language, making it accessible to a wider community. 

The rest of the paper is structured as follows. In the Section~\ref{sec.small_cycles} we present the proof of Theorems \ref{th1} and \ref{th2} and in the Section~\ref{sec.prisms} we present the details of the Hamiltonian cycle construction in the case of general $n$ and the proof of Theorem~\ref{th3}.

\section{Characterization of small cycles}\label{sec.small_cycles}
The general idea of the proofs is based on considering two antipodal vertices $\pi$ and $\tau$ of the $2\ell$--cycle, where $\ell =4,6$ and finding two non-intersecting $\ell$--paths between them (see Figure~\ref{fig.paths1}). For any positive integer $d\leqslant n-2$, the element $\pi_{i}$ of a permutation $\pi\in Sym_n$, where $1\leqslant i\leqslant n-d-1$, is called to be \textit{shifted} $d$ steps to the right from its original position, if $\tau_{i+d}=\pi_{i}$. This is achieved by applying the sequence of transpositions $\left( b_{i+1}\,b_{i+2}\,\dots\,b_{i+d}\right)$.

In case of each $\ell$, we subdivide the proof into cases, depending on the number of shifted elements of $\pi$ and the distance they are shifted along the $\pitau$-path. It is obvious that along the $\taupi$-path the same elements must be shifted back to the left to their original positions, and therefore we only consider shifting to the right. Since we consider cycles of small length, we evade looking into more complicated shifts along the paths.

\begin{figure}[t]
\centering
\includegraphics[scale=0.8]{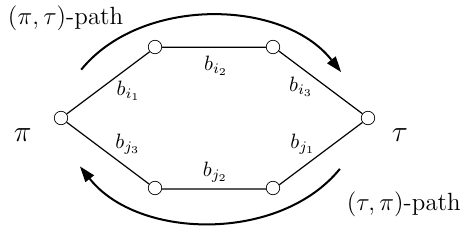}
\caption{Example of $\pitau$- and $\taupi$-paths in a 6--cycle.}
\label{fig.paths1}
\end{figure}

We call two transpositions $b_i$ and $b_j$ \textit{independent}, if $i\neq j-1,j+1$. Otherwise the transpositions are called \textit{dependent}. The main properties of such transpositions used in the proofs can be formulated as follows.
\begin{proposition}\label{prop.dependent_transpositions}
Let $\pi$ and $\tau$ be the vertices of the Bubble-sort graph $BS_n, n\geqslant 3$, and let the $\pitau$--path of length $d$, where $d\leqslant n-2$, is given by a sequence of successively dependent transpositions then there is no non-intersecting $\taupi$-path of length $d$.
\end{proposition}
\proof For any $j$, where $1\leqslant j\leqslant n-d$, the sequence of transpositions $(b_{j}\,b_{j+1}\dots\, b_{j+d-1})$ acts on $\pi$ as shifting of the element $j$ to position $j+d$. Since this is the only moved element, the backward $\taupi$-path of the same length inevitably has all the transpositions in reverse order. Hence, the path is unique. \qed

\subsection{Proof of Theorem~\ref{th1}}
We present now the complete description of $4$-cycles.
\begin{proof}
Since the graph is vertex-transitive, let $\pi=[12\dots n]$. For any two antipodal vertices $\pi$ and $\tau$ of a $4$-cycle, the length of the $\pitau$--path equals two. We prove the statement by considering all possible cases of the shifted elements of $\pi$ along the $\pitau$--path and describing the set of non-intersecting $\taupi$--paths. Obviously, the number of shifted elements should not be greater than two.

Shift of the one element is equivalent to having only two dependent transpositions on the $\pitau$--path and by Proposition~\ref{prop.dependent_transpositions} such a path is unique. If two elements are shifted, then $\pitau$--path consists of two independent transpositions $b_i\,b_j$, where $1\leqslant i \leqslant n-3$ and $i+3\leqslant j \leqslant n-1$. Since transpositions are independent, shuffling those will result in another path, which is a non-intersecting $\taupi$--path defined as $b_j\,b_i$.

\begin{figure*}[t]
\centering
\includegraphics[width = 0.8\textwidth]{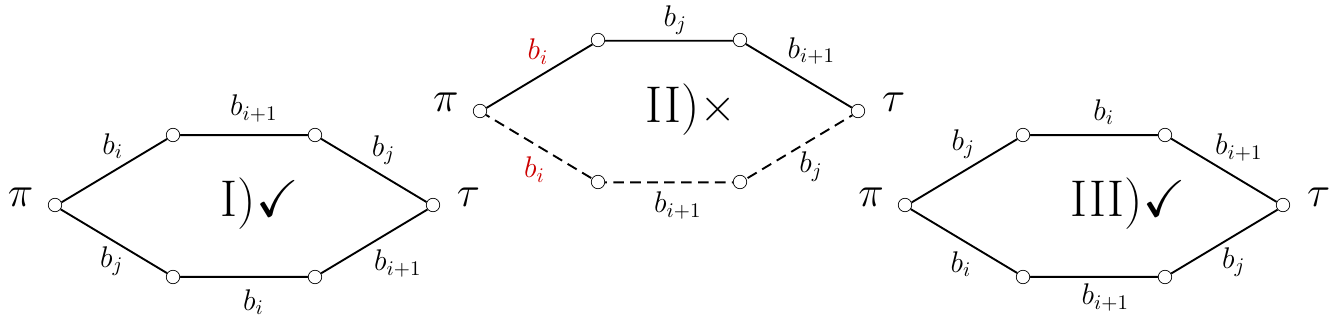}
\caption{Illustration for the order of two shifted elements. Orders I) and III) produce the same form; cycle in order II) is impossible.}
\label{fig.pathsC6twoshifted}
\end{figure*}

Altogether the canonical form of the only possible 4-cycles is given in \eqref{eq.c4}. It is straightforward to obtain the number of distinct cycles given by this family of forms. The number of possible pairs of indices is $(n-2)(n-3)/2$ and each form describes one distinct 4-cycle passing through a given vertex, thus in total there are $(n-2)(n-3)/2$ distinct 4-cycles passing through a given vertex and the total number of distinct $4$-cycles in $BS_n$ is $\frac{(n-2)(n-3)n!}{8}$. This completes the proof.
\end{proof}

\subsection{Proof of Theorem~\ref{th2}}
We present now the complete description of $6$-cycles.
\begin{proof}
Consider two antipodal vertices $\pi=[12\dots n]$ and $\tau$ of a $6$-cycle. The length of the $\pitau$--path is equal to three. We prove the statement by considering all possible cases of the shifted elements of $\pi$ along the $\pitau$--path and describing the set of non-intersecting $\taupi$--paths. The number of such elements should not be greater than three. This turns to considering the following cases.\vspace{2mm}

\textbf{Case 1: one shifted element.} Suppose the element $i$ is shifted to the position $i+3$. In order to do so, the sequence $b_{i}\,b_{i+1}\,b_{i+2}$ should be applied to $\pi$, which consists of successively dependent transpositions and by Proposition~\ref{prop.dependent_transpositions} such $\pitau$--path is unique. Hence, no 6-cycle is possible in this case.\vspace{2mm}

\textbf{Case 2: two shifted elements.} Suppose the element $i$ is shifted to the position $i+2$ and another element $j$ is shifted to $j+1$. In order to do so, the sequence $b_{i}\,b_{i+1}$ and the transposition $b_j$ should be applied to $\pi$. There are three possible orders of application (see Figure~\ref{fig.pathsC6twoshifted}):
\begin{equation}
\text{I)}\,\, b_{i}\,b_{i+1}\,b_j; \quad \text{II)}\,\, b_{i}\,b_j\,b_{i+1}; \quad \text{III)}\,\, b_j\,b_{i}\,b_{i+1}.
\end{equation}
There are two possibilities: either $b_j$ is independent from $b_{i}$ and $b_{i+1}$ or not.

If $b_j$ is independent from $b_{i}$ and $b_{i+1}$, then in order to shift the element $i$ back to the initial position, we need to apply the sequence $b_{i}\,b_{i+1}$ in inverse order on $\taupi$--path. By analyzing the possible ways of this operation (see Figure~\ref{fig.pathsC6twoshifted}), we find the only possible cycle form $b_j\,b_{i}\,b_{i+1}\,b_j\,b_{i+1}\,b_i$, which corresponds to \eqref{eq.c6_2} from the statement. Since $b_j$ is independent from $b_{i}$ and $b_{i+1}$, then either $i+3 \leqslant j$ or $j\leqslant i-2$. In the first case we have $1\leqslant i \leqslant n - 4 $ and $i+3 \leqslant j \leqslant n-1$. Taking the double sum over both these conditions we obtain $(n-3)(n-4)/2$. In the second case we have $1\leqslant j \leqslant n - 3$ and $j+2 \leqslant i \leqslant n-1$, which gives the another $(n-3)(n-4)/2$ forms. In total, this leaves us with exactly $(n-2)(n-3)$ distinct forms given by \eqref{eq.c6_2} and due to symmetries each form describes three distinct cycles passing through a given vertex, in total giving $3(n-2)(n-3)$ distinct cycles.

If $b_j$ is dependent on $b_{i}$ or $b_{i+1}$, then either $j=i$, $j=i+1$ or $j=i-1$.  Along the $\taupi$--path we must apply the sequence $b_{i}\,b_{i+1}$ in the inverse order, thus for each choice I), II) and III) of $\pitau$--path there exists a unique description of possible $\taupi$--path. In case $\pitau$--path is given by I), then for $j=i$ the non-intersecting $\taupi$--path is given as $b_{i+1}\,b_{i}\,b_{i+1}$, whereas $j=i-1$ and $j=i+1$ are impossible. In case $\pitau$--path is given by III), then for $j=i+1$ the non-intersecting $\taupi$--path is given as $b_{i}\,b_{i+1}\,b_{i}$, whereas $j=i-1$ and $j=i$ are impossible. The case of $\pitau$--path given by II) is also impossible. Therefore, we have the only form of a 6-cycle given by $b_{i+1}\,b_{i}\,b_{i+1}\,b_{i}\,b_{i+1}\,b_{i}$, which corresponds to~\eqref{eq.c6_1} in the statement. The number of distinct forms described by~\eqref{eq.c6_1} is $n-2$, and each form describes one cycle passing through a given vertex, in total giving $(n-2)$ distinct cycles.

\begin{figure}[b]
\centering
\includegraphics[scale=0.42]{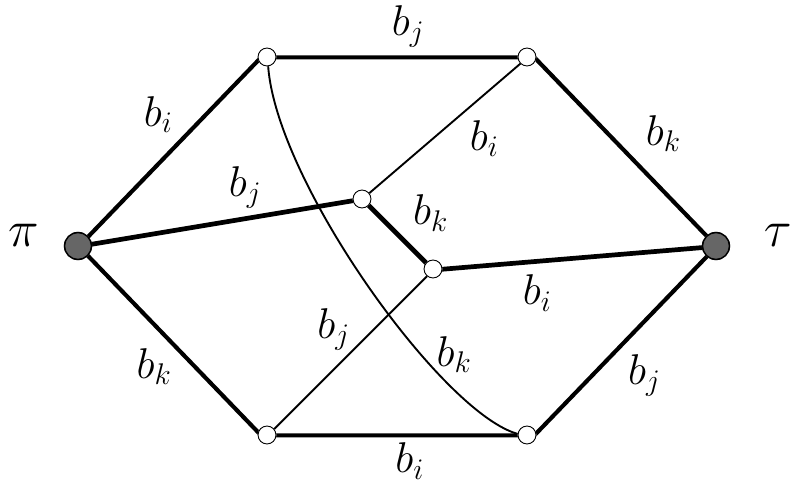}
\caption{Case of three shifted elements: three possible non-intersecting $\pitau$--paths with chords between them produce a graph isomorphic to a 3-cube.}
\label{fig.pathsC6threeshifted}
\end{figure}

\textbf{Case 3: three shifted elements.} Suppose the element $i$ is shifted to the position $i+1$ by $b_i$. Then we have two other elements $j$ and $k$ shifted to the positions $j+1$ and $k+1$ by independent transpositions $b_j$ and $b_k$ correspondingly. Suppose the $\pitau$-path is given by $b_i\,b_j\,b_k$. The sole restriction on the non-intersecting $\taupi$--path is that the two incident edges to $\pi$ and $\tau$ should be different. Therefore, we have two non-intersecting $\taupi$-paths: 1) $b_j\,b_k\,b_i$ and 2) $b_k\,b_i\,b_j$. Further, we note there are chords between pairs of paths $(b_i\,b_j\,b_k,b_j\,b_k\,b_i)$, $(b_j\,b_k\,b_i,b_k\,b_i\,b_j)$ and $(b_i\,b_j\,b_k,b_k\,b_i\,b_j)$ (see Figure~\ref{fig.pathsC6threeshifted}). Thus, combining all the non-intersecting pairs of paths, we obtain four distinct canonical forms of cycles described by~\eqref{eq.c6_3}-\eqref{eq.c6_6} in the statement. We note that these $6$-cycles appear only when $n\geqslant 6$.

\begin{figure*}[t]
\centering
\includegraphics[width=0.54\textwidth]{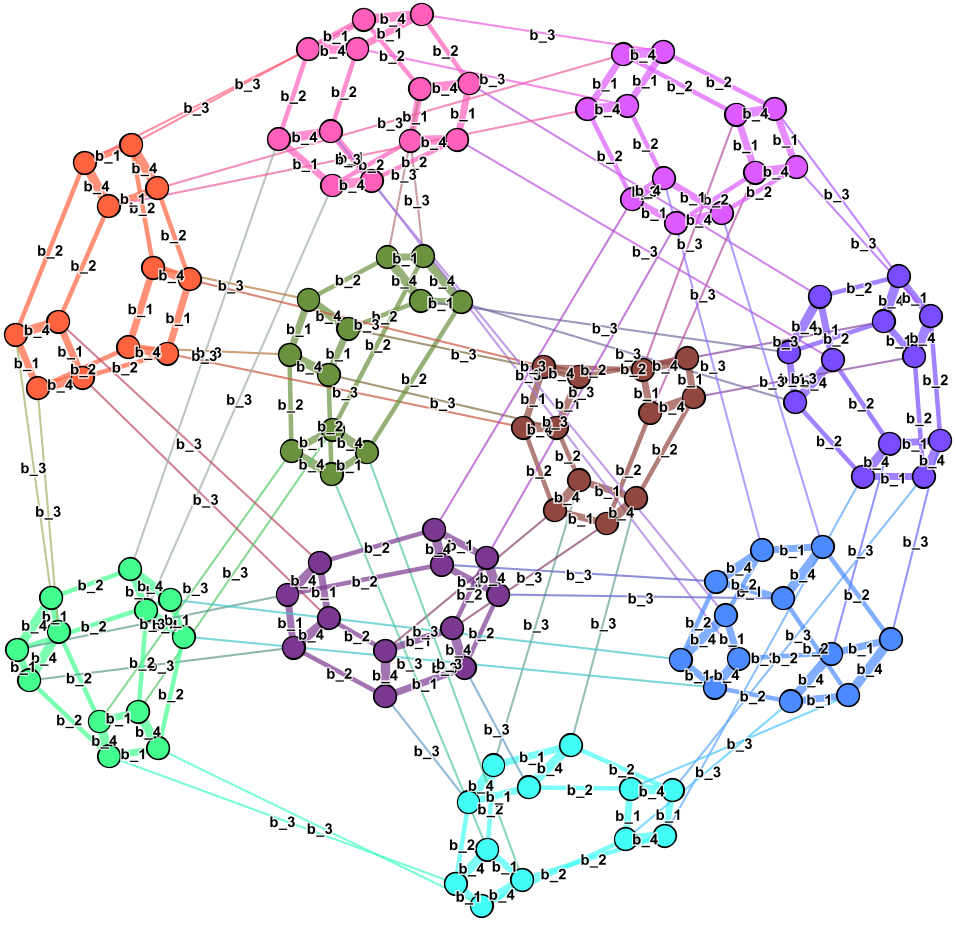}
\caption{The graph $BS_5$ with apparent $6$--prism structure marked with different colors.}
\label{fig.BS5_6-prisms}
\end{figure*}

One may see from the Figure~\ref{fig.pathsC6threeshifted} that paths form a subgraph of a 3-cube. Therefore any permutation of indices $i,j,k$ leaves us inside the same subgraph and in order to calculate the number of distinct forms $N(6)$ given by each of \eqref{eq.c6_3}-\eqref{eq.c6_6} it is only necessary to calculate the number of possible ordered triples $(i,j,k)$, such that $i \leqslant j-2 \leqslant k-4$. We note that for $n=6$ there is only one way to select three independent transpositions, thus $N(6)=1$. Increasing $n$ by one produces extra triples only by fixing the $k=n-1$, which is $\sum\limits_{j=3}^{n-3}\sum\limits_{i=1}^{j-2}1 = (n-4)(n-5)/2$, hence the following relation holds for $n\geqslant 7$:
$$N(n) = N(n-1) + \frac{(n-4)(n-5)}{2}.$$
From this relation, we immediately have the number $N(n) \\ =\frac{(n-3)(n-4)(n-5)}{6}$. Due to symmetries, each of the four forms~\eqref{eq.c6_3}-\eqref{eq.c6_6} describes three distinct cycles passing through a given vertex, all in total giving a contribution of $2(n-3)(n-4)(n-5)$ distinct cycles.

In total, the graph $BS_n$ contains 
\begin{multline}
\left((n-2)+3(n-2)(n-3)+2(n-3)(n-4)(n-5)\right)\frac{n!}{6} \\ = (2n^3-21n^2+80n-104)\frac{n!}{6}
\end{multline} 
cycles of length 6. This completes the proof.
\end{proof}

\section{Hamiltonian cycles in $BS_n$ based on generalised prisms}\label{sec.prisms}
In this section we give a description of a Hamiltonian cycle in the Bubble-sort graph $BS_n$ based on graphical prisms. We first show the existence of such cycles in the graph $BS_5$. Then the algorithm is further extended to the general case $n\geqslant 5$. We start the section with a few definitions and a technical lemma.

The parity of a permutation $\pi\in Sym_n$ may be defined as the parity of the number of transpositions in the canonical representation of $\pi$. Let us remind, that the Bubble-sort graph is bipartite, since each multiplication by a generating transposition changes parity of a permutation. \textit{The graph of an $n$--prism}, $n\geqslant 3$, is called the graph of a Cartesian product $C_n \times K_2$ of the cycle of length $n$ and the complete graph on two vertices. A graph is called \textit{Hamilton-connected} if any two vertices of a graph are connected by a Hamiltonian path. The following proposition is well known for prisms (see e.g.~\cite{Paulraja93}).

\begin{proposition}\label{prop.ham-connect-general_prism}
The graph of $n$--prism is Hamilton-con\-nected if and only if $n$ is odd.
\end{proposition}

\begin{figure*}[t]
\centering
\includegraphics[width=0.8\textwidth]{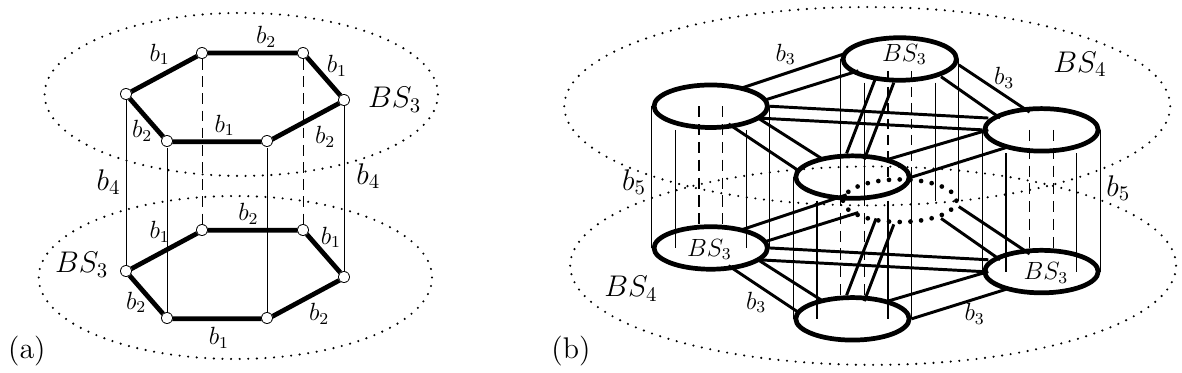}
\caption{a) Example of a $6$-prism (or $\GP{5}$); b) Schematic example of a generalised prism $\GP{6}$.}
\label{fig.prisms_examples}
\end{figure*}

The Hamiltonian cycle is constructed using generalized prisms. Remind that for any graph $G$ the \textit{generalised prism} 2-$G$ is defined as the Cartesian product $G\times K_2$. Let the graph $G$ have a maximal vertex cover by subgraphs, isomorphic to a subgraph $H\subset G$. Then, the \textit{factor graph} $\faktor{G}{H}=(V,E)$ is the graph whose nodes $V$ represent covering subgraphs and there is an edge $(V_1, V_2)\in E$ if in the original graph $G$ there is an edge between some nodes $u\in V_1$ and $v \in V_2$.

In the graph $BS_n$, where $n\geqslant 5$, we observe the presence of generalized prisms 2-$BS_{n-2}$, whose vertex set is given by
$$\left\{[\pi_1 \dots \pi_{n-2}\,i\,j] \cup [\pi_1 \dots \pi_{n-2}\,j\,i]:\,\pi_k\in \{1,\dots,n\}\backslash \{i,j\}\right\},$$
for fixed $i\neq j\in\{1,\dots,n\}$ and the edge set includes edges, that correspond to transpositions $b_1,\dots, b_{n-3}, b_{n-1}$. Further, note that $b_{n-1}$-edges only connect vertices of two copies of $BS_{n-2}$, since they permute only the last two elements of a permutation. Let us denote further for simplicity the 2-$BS_{n-2}$ as $\GP{n}$. On the Figure \ref{fig.prisms_examples} we show an example of the $\GP{6}$. By hierarchical structure, it is straightforward to show that there exists a maximal cover of $BS_{n}$ by $\GP{n}$, therefore there exists the factor graph $\Gamma_n = \faktor{BS_n}{\GP{n}}$. 

Let us describe the structure of the factor graph $\Gamma_n$. Since, the vertex set of any $\GP{n}$ contains all vertices with the last two elements fixed, then vertices of $\Gamma_n$ can be encoded by unordered pairs of elements $\{i,j\}$ from $\{1,2,\dots, n\}$ and a specific prism on vertices with last elements $i,\,j$ we denote as $\GP{n}(i,j)$. Following the connectivity between copies of $\GP{n}$, which is given by the edges corresponding to transposition $b_{n-2}$, two prisms $\GP{n}(i_1, j_1)$ and $\GP{n}(i_2, j_2)$ are connected iff $\{i_1, j_1\} \cap \{i_2, j_2\} = 1$. Hence $\Gamma_n$ is isomorphic to the \textit{Johnson graph} $J(n,2)$, which is proven to be Hamiltonian and even Hamilton-connected \cite{Alspach12}. For the sake of visual presentation we refer to Figure~\ref{fig.BS5_6-prisms} with an example of the covering of the graph $BS_5$ by prisms $\GP{5}(2,4)$

Furthermore, we note that any two factor vertices of $\GP{n}$ are connected through $\GP{n-1}$, meaning that for any appropriate $k$ all vertices of one $\GP{n-1}(k,i,j)\subset \GP{n}(i,j)$ on the vertices of type $[\dots k\,i\,j]$ with fixed last elements $k, i, j$ are connected to the vertices of $\GP{n-1}(i, k, j)\subset \GP{n}(k,j)$ on the vertices of type $[\dots i\,k\,j]$ with fixed last elements $i,k,j$.

\subsection{Proof of Theorem \ref{th3}}
\begin{proof}
We use the lifting method to construct the Hamiltonian cycle. We showed above that the factor graph is nice enough, thus we need to prove a certain statement about the connectivity of $\GP{n}$.

\begin{proposition}\label{prop.2-bsn-2_ham_path}
In the graph of the $\GP{n}\subset BS_n$, where $n\geqslant 5$, for a given vertex $\pi$ and any vertex $\tau$, such that the parity $\omega(\pi)\neq \omega(\tau)$, there is a Hamiltonian path with end vertices $\pi$ and $\tau$.
\end{proposition}
\begin{proof}
By definition, the prism $\GP{n}$ is represented as two copies of $BS^1_{n-2}$ and $BS^2_{n-2}$, put one above the other and connected by $b_{n-1}$-edges between nodes of respective copies. Since the graph $BS_{n}$ is Hamiltonian for any $n\geqslant 3$, let us fix some Hamiltonian cycles $H^1_{n-2}\in BS^1_{n-2}$ and $H^2_{n-2}\in BS^2_{n-2}$ of similar structure, i.e. such that if an edge $(\pi_1,\pi_2)\in H^1_{n-2}$, then it implies that $(\pi_1\,b_{n-1},\pi_2\,b_{n-1})\in H^2_{n-2}$.

The graph $BS_{n}$ is bipartite, thus any path between two nodes of the opposite parity has odd length and between the nodes of the same parity the length is even. We shall use this fact to construct the desired Hamiltonian path between $\pi$ and $\tau$. Since the graph is vertex-transitive, let us fix the vertex $\pi$ in the copy $BS^1_{n-2}$ and consider two cases of placement of the vertex $\tau$. \vspace{2mm}

\begin{figure*}[t]
\centering
\includegraphics[width=0.75\textwidth]{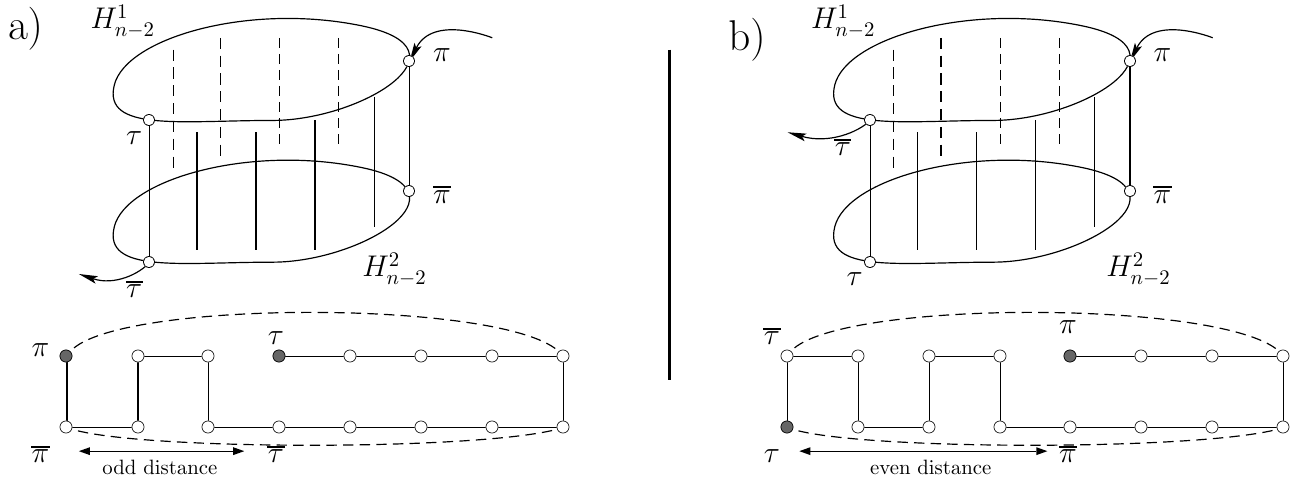}
\caption{Schematic structure of the Hamiltonian path construction in prisms from Proposition~\ref{prop.2-bsn-2_ham_path}.}
\label{fig.prism_ham_cycle}
\end{figure*}

\textbf{Case 1.} Let $\tau$ be in the same copy $BS^1_{n-2}$. Then, denote $\overline{\pi} = \pi b_{n-1}$ and $\overline{\tau} = \tau b_{n-1}$. The nodes $\overline{\pi}$ and $\overline{\tau}$ are of different parity, thus the $\pioverlinetauoverline$-path has odd length. Therefore, we can start the Hamiltonian path at $\pi$, then proceed in alternating way as shown on Figure~\ref{fig.prism_ham_cycle} until we reach $\overline{\tau}$. Then we continue the path in the copy $BS^2_{n-2}$ until we reach the direct neighbour of the vertex $\overline{\pi}$ in the $H^2_{n-2}$. Moving back to the $BS^1_{n-2}$ and following the unvisited part of $H^1_{n-2}$ we reach the vertex $\tau$, thus obtaining the desired Hamiltonian path.

\textbf{Case 2.} Let $\tau$ be in the different copy $BS^2_{n-2}$. Let us again denote $\overline{\pi} = \pi b_{n-1}$ and $\overline{\tau} = \tau b_{n-1}$. The nodes $\pi$ and $\overline{\tau}$ are of the same parity, thus the $\pitauoverline$-path has even length. Therefore, we can start the Hamiltonian path at $\pi$, then proceed along the $H^1_{n-2}$ until we reach the direct neighbour of $\overline{\tau}$ as shown on Figure~\ref{fig.prism_ham_cycle}. Then we move to the copy $BS^2_{n-2}$ and continue along the path $H^2_{n-2}$ backwards (without passing through $\tau$) until we reach the vertex $\overline{\pi}$ in the $H^2_{n-2}$. The distance between the direct neighbour of $\overline{\pi}$ and $\tau$ is odd, since the length of the $H^2_{n-2}$ is even and the traversed segment of $H^1_{n-2}$ and $H^2_{n-2}$ is of odd length, therefore we can continue in the alternating way until we reach the vertex $\tau$ (see Figure \ref{fig.prism_ham_cycle}), thus obtaining the desired Hamiltonian path.
\end{proof}

We now turn to the proof of the main theorem. The Hamiltonian cycle in $\Gamma_n$ defines the traversal sequence $\GP{n}^{(j_{1})} \rightarrow \GP{n}^{(j_{2})} \rightarrow \dots \rightarrow \GP{n}^{(j_{m})} \rightarrow \GP{n}^{(j_{1})}$ of generalized prisms and let us construct the lifted cycle in the original graph. Let us fix the vertex incident to the edge between $\GP{n}^{(j_{1})}$ and $\GP{n}^{(j_{2})}$ as a starting vertex $\pi_{(j_{1})}$ with parity $\omega$. Traversing this edge will change the parity, and the adjacent vertex $\tau_{(j_{2})}$ in $\GP{n}^{(j_{2})}$ will have an opposite parity $\overline{\omega}$. By Proposition~\ref{prop.2-bsn-2_ham_path}, there is a Hamiltonian path that ends in a vertex $\pi_{(j_{2})}$ with opposite parity $\omega$ which is adjacent to $\tau_{(j_{3})}$ in $\GP{n}^{(j_{3})}$. Traversing to $\tau_{(j_{3})}$ changes the parity again to $\overline{\omega}$, which was the case for $\tau_{(j_{2})}$. Proceeding in this way, the parity of the starting and ending vertices in the prisms will be different, hence when the construction arrives back to the starting prism $\GP{n}^{(j_{1})}$, the ending vertex $\tau_{(j_{1})}$ will have parity $\overline{\omega}$, opposite from the parity of $\pi_{(j_{1})}$. Therefore, this vertex is at the odd distance from $\pi_{(j_{1})}$ and connecting them with a Hamiltonian path will result in a Hamiltonian cycle in the original graph $BS_n$. This finishes the proof of the Theorem~\ref{th3}.
\end{proof}

\section*{Acknowledgement}
The authors thank Zahary Hamaker for referring us to permutahedra, which have direct relation to the graph under our consideration and Valeriy Bardakov for providing insights regarding permutahedra. We thank the anonymous reviewers for their valuable suggestions that helped to improve the quality of the manuscript. This work was supported partially by the grants 17-51-560008, 18-01-00353 and 19-01-00682 of the Russian Foundation for Basic Research.

\bibliographystyle{elsarticle-num-names}
\bibliography{references}

\newpage \onecolumn
\section*{Appendix}

In the Appendix we present the table of all Hamiltonian paths between any two nodes at odd distance in the 6-prism $P_6(i,k)$.

\begin{table*}[ht!]
\centering
\begin{tabular}{C{0.8cm} | L{7cm} L{7cm}}

\rule[-1ex]{0pt}{2.5ex}  & \multicolumn{2}{c}{All possible Hamiltonian paths with end vertices $(\pi,\tau)$}  \\
\hline
\rule[-1ex]{0pt}{2.5ex}  $\pi\,b_{i-1}$ & $[b_{i}\,b_{i-1}\,b_{i}\,b_{i-1}\,b_{k}\,b_{i-1}\,b_{i}\,b_{i-1}\,b_{i}\,b_{i-1}\,b_{k}]$, $[b_{i}\,b_{i-1}\,b_{i}\,b_{k}\,b_{i}\,b_{i-1}\,b_{i}\,b_{i-1}\,b_{i}\,b_{k}\,b_{i}]$, $[b_{i}\,b_{i-1}\,b_{k}\,b_{i-1}\,b_{i}\,b_{i-1}\,b_{i}\,b_{i-1}\,b_{k}\,b_{i-1}\,b_{i}]$ & $[b_{i}\,b_{k}\,b_{i}\,b_{i-1}\,b_{i}\,b_{i-1}\,b_{i}\,b_{k}\,b_{i}\,b_{i-1}\,b_{i}]$, $[b_{k}\,b_{i-1}\,b_{i}\,b_{i-1}\,b_{i}\,b_{i-1}\,b_{k}\,b_{i-1}\,b_{i}\,b_{i-1}\,b_{i}]$, $[b_{k}\,b_{i}\,b_{k}\,b_{i-1}\,b_{k}\,b_{i}\,b_{k}\,b_{i-1}\,b_{k}\,b_{i}\,b_{k}]$ \\
\hline
\rule[-1ex]{0pt}{2.5ex} $\pi\,b_{i}$ & $[b_{i-1}\,b_{i}\,b_{i-1}\,b_{i}\,b_{k}\,b_{i}\,b_{i-1}\,b_{i}\,b_{i-1}\,b_{i}\,b_{k}]$, $[b_{i-1}\,b_{i}\,b_{i-1}\,b_{k}\,b_{i-1}\,b_{i}\,b_{i-1}\,b_{i}\,b_{i-1}\,b_{k}\,b_{i-1}]$, $[b_{i-1}\,b_{i}\,b_{k}\,b_{i}\,b_{i-1}\,b_{i}\,b_{i-1}\,b_{i}\,b_{k}\,b_{i}\,b_{i-1}]$, & $[b_{i-1}\,b_{k}\,b_{i-1}\,b_{i}\,b_{i-1}\,b_{i}\,b_{i-1}\,b_{k}\,b_{i-1}\,b_{i}\,b_{i-1}]$, $[b_{k}\,b_{i-1}\,b_{k}\,b_{i}\,b_{k}\,b_{i-1}\,b_{k}\,b_{i}\,b_{k}\,b_{i-1}\,b_{k}]$, $[b_{k}\,b_{i}\,b_{i-1}\,b_{i}\,b_{i-1}\,b_{i}\,b_{k}\,b_{i}\,b_{i-1}\,b_{i}\,b_{i-1}]$ \\
\hline
\rule[-1ex]{0pt}{2.5ex} $\pi\,b_{k}$ & $[b_{i-1}\,b_{i}\,b_{i-1}\,b_{i}\,b_{i-1}\,b_{k}\,b_{i-1}\,b_{i}\,b_{i-1}\,b_{i}\,b_{i-1}]$, $[b_{i-1}\,b_{k}\,b_{i}\,b_{k}\,b_{i-1}\,b_{k}\,b_{i}\,b_{k}\,b_{i-1}\,b_{k}\,b_{i}]$, & $[b_{i}\,b_{i-1}\,b_{i}\,b_{i-1}\,b_{i}\,b_{k}\,b_{i}\,b_{i-1}\,b_{i}\,b_{i-1}\,b_{i}]$, $[b_{i}\,b_{k}\,b_{i-1}\,b_{k}\,b_{i}\,b_{k}\,b_{i-1}\,b_{k}\,b_{i}\,b_{k}\,b_{i-1}]$  \\
\hline
\rule[-1ex]{0pt}{2.5ex} $\pi\cdot b_{i-1}$ $b_{i}\,b_{i-1}$ & $[b_{i-1}\,b_{i}\,b_{k}\,b_{i}\,b_{i-1}\,b_{i}\,b_{k}\,b_{i-1}\,b_{k}\,b_{i}\,b_{k}]$, $[b_{i-1}\,b_{k}\,b_{i-1}\,b_{i}\,b_{k}\,b_{i-1}\,b_{k}\,b_{i}\,b_{i-1}\,b_{k}\,b_{i-1}]$, $[b_{i}\,b_{i-1}\,b_{k}\,b_{i-1}\,b_{i}\,b_{i-1}\,b_{k}\,b_{i}\,b_{k}\,b_{i-1}\,b_{k}]$, & $[b_{i}\,b_{k}\,b_{i}\,b_{i-1}\,b_{k}\,b_{i}\,b_{k}\,b_{i-1}\,b_{i}\,b_{k}\,b_{i}]$, $[b_{k}\,b_{i-1}\,b_{k}\,b_{i}\,b_{k}\,b_{i-1}\,b_{i}\,b_{i-1}\,b_{k}\,b_{i-1}\,b_{i}]$, $[b_{k}\,b_{i}\,b_{k}\,b_{i-1}\,b_{k}\,b_{i}\,b_{i-1}\,b_{i}\,b_{k}\,b_{i}\,b_{i-1}]$ \\
\hline
\rule[-1ex]{0pt}{2.5ex} $\pi\cdot b_{i-1}$ $b_{i}\,b_{k}$ & $[b_{i-1}\,b_{i}\,b_{i-1}\,b_{k}\,b_{i}\,b_{k}\,b_{i-1}\,b_{k}\,b_{i}\,b_{i-1}\,b_{i}]$, $[b_{i-1}\,b_{k}\,b_{i-1}\,b_{i}\,b_{k}\,b_{i-1}\,b_{k}\,b_{i}\,b_{k}\,b_{i-1}\,b_{k}]$, $[b_{i}\,b_{i-1}\,b_{i}\,b_{i-1}\,b_{i}\,b_{k}\,b_{i-1}\,b_{i}\,b_{i-1}\,b_{i}\,b_{i-1}]$, $[b_{i}\,b_{i-1}\,b_{i}\,b_{k}\,b_{i}\,b_{i-1}\,b_{i}\,b_{i-1}\,b_{k}\,b_{i}\,b_{k}]$, & $[b_{i}\,b_{i-1}\,b_{k}\,b_{i-1}\,b_{i}\,b_{i-1}\,b_{k}\,b_{i}\,b_{i-1}\,b_{k}\,b_{i-1}]$, $[b_{i}\,b_{k}\,b_{i}\,b_{i-1}\,b_{k}\,b_{i}\,b_{i-1}\,b_{i}\,b_{k}\,b_{i}\,b_{i-1}]$, $[b_{k}\,b_{i-1}\,b_{k}\,b_{i}\,b_{i-1}\,b_{i}\,b_{i-1}\,b_{k}\,b_{i-1}\,b_{i}\,b_{i-1}]$, $[b_{k}\,b_{i}\,b_{k}\,b_{i-1}\,b_{k}\,b_{i}\,b_{k}\,b_{i-1}\,b_{i}\,b_{k}\,b_{i}]$ \\
\hline
\rule[-1ex]{0pt}{2.5ex} $\pi\cdot b_{i}$ $b_{i-1}\,b_{k}$ &    $[b_{i-1}\,b_{i}\,b_{i-1}\,b_{i}\,b_{i-1}\,b_{k}\,b_{i}\,b_{i-1}\,b_{i}\,b_{i-1}\,b_{i}]$, $[b_{i-1}\,b_{i}\,b_{i-1}\,b_{k}\,b_{i-1}\,b_{i}\,b_{i-1}\,b_{i}\,b_{k}\,b_{i-1}\,b_{k}]$, $[b_{i-1}\,b_{i}\,b_{k}\,b_{i}\,b_{i-1}\,b_{i}\,b_{k}\,b_{i-1}\,b_{i}\,b_{k}\,b_{i}]$,  $[b_{i-1}\,b_{k}\,b_{i-1}\,b_{i}\,b_{k}\,b_{i-1}\,b_{i}\,b_{i-1}\,b_{k}\,b_{i-1}\,b_{i}]$, & $[b_{i}\,b_{i-1}\,b_{i}\,b_{k}\,b_{i-1}\,b_{k}\,b_{i}\,b_{k}\,b_{i-1}\,b_{i}\,b_{i-1}]$, $[b_{i}\,b_{k}\,b_{i}\,b_{i-1}\,b_{k}\,b_{i}\,b_{k}\,b_{i-1}\,b_{k}\,b_{i}\,b_{k}]$, $[b_{k}\,b_{i-1}\,b_{k}\,b_{i}\,b_{k}\,b_{i-1}\,b_{k}\,b_{i}\,b_{i-1}\,b_{k}\,b_{i-1}]$, $[b_{k}\,b_{i}\,b_{k}\,b_{i-1}\,b_{i}\,b_{i-1}\,b_{i}\,b_{k}\,b_{i}\,b_{i-1}\,b_{i}]$ \\

\end{tabular}
\caption{The list of Hamiltonian paths from the source node $\pi$ in the 6-prism $P_6(i,k)$ to any possible node with the odd distance.}
\label{tab.ham_paths_6_prism}
\end{table*}

\end{document}